\documentclass[12pt]{article}

\usepackage{amsmath,amsthm,amsfonts,amsxtra,amssymb,upref}
\usepackage[mathscr]{eucal}
\usepackage[usenames,dvipsnames]{color}
\usepackage[a4paper,margin=2.5cm]{geometry}

\usepackage{mathtools}
\usepackage{enumerate}

\usepackage{authblk}

\numberwithin{equation}{section}

\theoremstyle{plain}
\newtheorem{thm}{Theorem}[section]
\newtheorem{cor}[thm]{Corollary}
\newtheorem{lem}[thm]{Lemma}
\newtheorem{prop}[thm]{Proposition}

\theoremstyle{definition}
\newtheorem{assumption}[thm]{Assumption}

\newtheorem{rem}[thm]{Remark}

\newcommand{\B}{\mathscr{B}}

\newcommand{\eps}{\varepsilon}
\newcommand{\N}{\mathbb{N}}
\newcommand{\R}{\mathbb{R}}

\newcommand{\sn}{\mathcal{S}}

\newcommand{\C}{\mathbb{C}}

\newcommand{\K}{\mathscr{K}}

\newcommand{\U}{\mathcal{U}}
\newcommand{\idd}{\mathbf{1}}
\newcommand{\id}{I}

\newcommand{\x}{\langle x \rangle }
\newcommand{\cH}{\mathcal{H}}
\newcommand{\jap}[1]{\langle{#1}\rangle}

\newcommand{\vp}{\varphi}
\newcommand{\w}{{\rm w}}
\newcommand{\cG}{\mathcal{G}}

\DeclareMathOperator{\Div}{div}
\DeclareMathOperator{\dom}{dom}
\DeclareMathOperator{\im}{Im}
\DeclareMathOperator{\ran}{ran}
\DeclareMathOperator{\re}{Re}

\DeclareMathOperator{\sgn}{sgn}

\newcommand{\cS}{\mathcal{S}(\R^3)}
\newcommand{\abs}[1]{\lvert{#1}\rvert}

\newcommand{\norm}[1]{\lVert{#1}\rVert}
\newcommand{\ket}[1]{\lvert{#1}\rangle}
\newcommand{\bra}[1]{\langle{#1}\rvert} 
\newcommand{\ip}[2]{\langle{#1},{#2}\rangle}

\newcommand{\cO}{\mathcal{O}}

\DeclareMathOperator{\rt}{curl}
\DeclareMathOperator{\nul}{null}
\DeclareMathOperator{\rank}{rank}
\DeclareMathOperator{\diag}{diag}

\newcommand{\myW}{\mathsf{W}}

\newcommand{\myw}{\mathsf{w}}
\newcommand{\myU}{\mathsf{U}}

\newcommand{\oneone}{\begin{bsmallmatrix} 1\\1\end{bsmallmatrix}}
\DeclareMathOperator{\myspan}{span}

\title{Resolvent expansions of 3D magnetic Schr\"odinger operators and Pauli operators}
\author{Arne Jensen\footnote{Department of Mathematical Sciences, Aalborg
 University, Skjernvej 4A, DK-9220 Aalborg~\O{}, Denmark, 
\texttt{matarne@math.aau.dk}} \ and  Hynek Kova\v{r}\'{\i}k\footnote{DICATAM, Sezione di Matematica, Università degli studi di Brescia, Via Branze 38, Brescia, 25123, Italy, \texttt{hynek.kovarik@unibs.it}}}

\begin{document}

\maketitle

\smallskip

\noindent {\bf MSC 2020:}  35Q40, 35P05, 35J10

\begin{abstract}
We obtain asymptotic resolvent expansions at the threshold of the essential spectrum for magnetic Schr\"odinger and Pauli operators in dimension three. These operators 
are treated as perturbations of the Laplace operator in $L^2(\R^3)$ and $L^2(\R^3;\C^2)$, respectively. 
The main novelty of our 
approach is to  show that the relative perturbations, which are first order differential operators, can be factorized in suitably chosen auxiliary spaces. This allows us to derive 
the desired  asymptotic expansions of the resolvents around zero. We then calculate their leading and sub-leading terms explicitly. 
Analogous factorization schemes for more general perturbations, including e.g.~finite rank perturbations, are discussed as well. 
\end{abstract}

\section{Introduction}
\label{sec-intro}

The purpose of this paper is to prove asymptotic expansions around the threshold zero of the resolvents of magnetic Schr\"{o}dinger operators and Pauli operators in dimension three. Besides being of interest on their own, resolvent expansions are also important for treating the low energy part in the proof  of dispersive estimates for the operators we consider. As far as we know the results obtained here are the first to treat in detail all possible cases for the threshold zero. 

Using the notation $P=-i\nabla$, 
the magnetic Schr\"{o}dinger operator is the operator
\begin{equation}
H=(P-A)^2+V\quad\text{on $L^2(\R^3)$}.
\end{equation}
Here $A\colon\R^3\to\R^3$ is the magnetic vector potential and 
$V\colon\R^3\to\R$ the electrostatic potential. We assume that both $V$ and $A$ are bounded and decay sufficiently fast. More precisely, in the latter case we assume that the magnetic field decays fast enough and show 
that a vector potential $A$ can be constructed in such way as to satisfy the required decay conditions, cf.~Lemma \ref{A-choice}.

We consider the resolvent $R(z)=(H-zI)^{-1}$. It is convenient to change variable in the resolvent to $\kappa$, where for $z\in\C\setminus(-\infty,0]$ we take $\kappa=-i\sqrt{\kappa}$, where
$\im \sqrt{z}>0$, such that $z=-\kappa^2$. Then we write $R(\kappa)=(H+\kappa^2I)^{-1}$.

To analyze the behavior of $R(\kappa)$ around the threshold zero the first step is to analyze the structure of solutions to $Hu=0$ in the weighted Sobolev spaces $H^{1,-s}$, $\frac12<s\leq\frac32$. See Section~\ref{set-up} for the definition of these spaces.
The solutions can be classified as follows. Assume that there exist $N$ linearly independent solutions to $Hu=0$ in $H^{1,-s}$, $\frac12<s\leq\frac32$. These solutions can be chosen in such a way that at most one solution $u\notin L^2(\R^3)$. The remaining $N-1$ solutions are eigenfunctions of $H$. This is the same classification as in the case $A=0$, see~\cite{JK}.

The resolvent expansions are in the topology of the bounded operators from $H^{-1,s}$ to $H^{1,-s'}$, for admissible values of $s$, $s'$.
In  the generic case there are no non-zero solutions to $Hu=0$ in $H^{1,-s}$, $\frac12<s\leq\frac32$, and then zero is said to be a regular point for $H$. The leading part of asymptotic expansion takes the form
\begin{equation}
R(\kappa)=F_0+\kappa F_1+\cO(\kappa^2)\quad\text{as $\kappa\to0$}
\end{equation}
in the topology of bounded operators from $H^{-1,s}$ to $H^{1,-s'}$, $s,s'>\frac52$.

If there exist non-zero solutions to $Hu=0$ in $H^{1,-s}$, $\frac12<s\leq\frac32$, zero is said to be an exceptional point. In this case the leading part of the asymptotic expansion takes the form
\begin{equation}
R(\kappa)=\kappa^{-2}F_{-2}+\kappa^{-1}F_{-1}+\cO(1)
\quad\text{as $\kappa\to0$}
\end{equation}
in the topology of bounded operators from $H^{-1,s}$ to $H^{1,-s'}$, $s,s'>\frac92$.

More precisely, there are three exceptional cases. In the first exceptional case there exists only one (up to normalization) solution to $Hu=0$ in $H^{1,-s}$, $\frac12<s\leq\frac32$, such that $u\notin L^2(\R^3)$. This is the \emph{zero resonance} case. In this case $F_{-2}=0$ and $F_{-1}=\ket{\psi_c}\bra{\psi_c}$, where $\psi_c$ is a normalization of the non-zero solution $u$. In the second exceptional case all solutions to $Hu=0$ in $H^{1,-s}$, $\frac12<s\leq\frac32$, lie in $L^2(\R^3)$ and zero is an eigenvalue of $H$. In this case $F_{-2}=P_0$, the eigenprojection of eigenvalue zero of $H$. The operator $F_{-1}$ is of rank at most $3$. It is described more precisely in Theorem~\ref{thm53}. The third exceptional case is the one where one has both a zero resonance and at least one zero eigenvalue. With the right choice of zero resonance function $\psi_c$ we have $F_{-2}=P_0$, and $F_{-1}$ is the sum of the coefficients in the first and second exceptional cases. 
See Theorems~\ref{thm-regular} and~\ref{thm53} for the full statements of the results.

Next we obtain similar resolvent expansions for the Pauli operator
\begin{equation*}
H_P=\bigl(\sigma\cdot(P-A)\bigr)^2+V\idd_2
=(P-A)^2\idd_2+\sigma\cdot B+V\idd_2,
\end{equation*}
where $\sigma=(\sigma_1,\sigma_2,\sigma_3)$ denotes the Pauli matrices and $\idd_2$ the $2\times2$ identity matrix. The operator is defined on $L^2(\R^3;\C^2)$.
We decompose it as
\begin{equation}
H_P=-\Delta\idd_2+W_P\, ,
\end{equation}
where
\begin{equation*}
W_P=\bigl(-P\cdot A -A\cdot P +\abs{A}^2\bigr) \idd_2
+ V\idd_2+\sigma\cdot B.
\end{equation*}
Then we can obtain a classification of the point zero in the spectrum of $H_P$. 
The singularity structure is the same as for the magnetic Schr\"{o}dinger operator, however in this case the zero resonance can have multiplicity two,
see Theorems \ref{thm-regular-pauli} and \ref{thm-ex-pauli}.

The proofs of these results are obtained by taking the results on the resolvent expansion of $-\Delta$ from~\cite{JK} and combining them with the factored resolvent technique from~\cite{jn}, adapted to the two cases considered here. The main point here is that we can write the magnetic Schr\"{o}dinger operator as
\begin{equation*}
H=-\Delta+W\quad\text{with} \quad W= -P\cdot A-A\cdot P+A^2+V,
\end{equation*}
and we can factor the perturbation as $W=w^*Uw$.
Let $\beta>0$ be the decay rate of the potentials. Then
$w\colon H^{1,-\beta/2}(\R^3)\to\K$ ($\K$ an auxiliary space) and $U$ is a self-adjoint and unitary operator on $\K$.
 See Section~\ref{sec-factored} for the details. A similar factorization holds from $W_P$ in the Pauli operator case. Once the factorization is in place, the scheme from~\cite{jn} can be applied and leads to the resolvent expansion results.

It should be noted that the factorization method developed in this paper can be applied not only to perturbations arising from magnetic Hamiltonians, but to all perturbations represented by self-adjoint
first order differential operators, see Remark~\ref{rem-first-order} for more details. 

As applications of the resolvent expansions of $H$ and $H_P$ around zero we obtain some further results.
First, we consider the case $V\geq0$ for the magnetic Schr\"{o}dinger operator and show that the point zero is a regular point. See Corollary~\ref{cor-regular}. Second, for the Pauli operator we consider the case $V=0$ and show that there are no zero resonances,  cf.~Lemma~\ref{lem-no-resonance}. Moreover, we establish the connection between our results and the criterion for zero eigenvalues obtained in~\cite{be,bel,bvb}, see Proposition~\ref{prop-zero-mode}.

Resolvent expansions have a long history. We will not give a full account, but limit ourselves to the following remarks. Results on Schr\"{o}dinger operators in $L^2(\R^3)$ were obtained in~\cite{JK}. In particular, the classification of the point zero used here was introduced in this paper. 
All dimensions and general perturbations, including first order differential operators, were considered in~\cite{mu}, but the coefficients of the resolvent expansions obtained there 
are given only implicitly as solutions of certain operator equations.

After these two papers there are many papers obtaining resolvent expansions in many different contexts. In the two-dimensional setting, resolvent expansions of magnetic Schr\"odinger operators, for the generic case, and of purely magnetic Pauli operators  were established in~\cite{kov1,kov2}.

However,  in dimension three, very few papers treated the case of magnetic Schr\"{o}dinger operators and none of them Pauli operators, as far as we know. 
Partial results in the generic case for magnetic 
Schr\"{o}dinger operators were obtained in~\cite{kk}. Behavior of the resolvent at threshold, again in the generic case,
was studied also in~\cite{egs}, where 
Strichartz estimates for magnetic Schr\"{o}dinger operators are proved.

The paper is organized as follows. In Section~\ref{set-up} we introduce notation and the basic set-up for magnetic Schr\"{o}dinger operators. In Section~\ref{sec-free-exp} we recall some results on the free resolvent from~\cite{JK}. Section~\ref{sec-factored} is devoted to the factored resolvent equation. We derive a number of properties of the operators entering into this factorization. In Section~\ref{sec-main} we state the main results on resolvent expansions for magnetic Schr\"{o}dinger operators. We limit the statements to the ones giving the singularity structure at threshold zero. In Section~\ref{sec-pauli} we state the results on the Pauli operator. In the final Section~\ref{sec-general} we briefly explain how to obtain a factorization of a general perturbation, thus allowing one to treat for example finite rank perturbations of a magnetic Schr\"{o}dinger operator.

\section{The set-up}\label{set-up}
We will consider magnetic Schr\"{o}dinger operators in $\R^3$. 
Let $B$ be a magnetic field in $\R^3$ and let $A\colon\R^3\to\R^3$  be an associated vector potential satisfying $\rt A=B$. Moreover, let $V\colon\R^3 \to \R$ be a scalar electric field. 
We consider the magnetic Schr\"odinger operator
\begin{equation} \label{schr-op}
H = (P -A)^2 +V,\quad \text{where $P = -i\nabla$},
\end{equation}
on $L^2(\R^3)$. 
Its resolvent is denoted by
\begin{equation*}
R(z)=  (H-zI)^{-1}. 
\end{equation*}
Our goal is to obtain asymptotic expansions of this resolvent around the threshold zero of $H$. 
These expansions are valid in the topology of bounded operators between weighted Sobolev spaces.

We recall the definition of the weighted Sobolev spaces. Let $\x = (1+\abs{x}^2)^{1/2}$. On the Schwartz space $\cS$ define a norm
\begin{equation} \label{hms-norm}
\norm{u}_{H^{k,s}} = \norm{\x^s (1- \Delta)^{k/2} u}_{L^2(\R^3)},  \quad \text{$k\in\R$, $s\in\R$}.
\end{equation}
The completion of $\cS$ with this norm is the weighted Sobolev space, denoted by $H^{k,s}(\R^3)$. In the sequel we abbreviate this notation to $H^{k,s}$. The same holds for other spaces defined on $\R^3$.
Obviously, $H^{0,0}=L^2(\R^3)$. The inner product $\ip{\cdot}{\cdot}$ on $L^2$ extends to a duality between $H^{k,s}$ and $H^{-k,-s}$.
The bounded operators from $H^{k,s}$ to $H^{k',s'}$ are denoted by
\begin{equation*}
\B(k,s;k',s') = \B(H^{k,s}; H^{k',s'}) 
\end{equation*}
and this space is equipped with the operator norm.

For later use we note the following property.
Let $s_j\in\R$, $j=1,2$, with $s_1\leq s_2$ and $k\in\R$. Then we have the continuous embedding
\begin{equation}\label{embed}
H^{k,s_2}\hookrightarrow H^{k,s_1}.
\end{equation}

It is convenient to use the notation
\begin{equation*}
 H^{k,s+0} = \bigcup_{s<r}  H^{k,r} , \qquad  
H^{k,s-0} = \bigcap_{r<s}  H^{ k,r}.
\end{equation*}
However, we do not introduce topologies on these spaces. They are considered only as algebraic vector spaces.

Let us now state the assumptions on $B$ and $V$, and explain our choice of vector potential $A$.

\begin{assumption}\label{ass-BV}
Let $\beta>2$.
Let $V\colon\R^3\to\R$ satisfy 
\begin{equation}
\abs{V(x)}\lesssim\x^{-\beta},\quad x\in\R^3.
\end{equation}
Let $B\colon\R^3\to\R^3$ be continuously differentiable, such that $\nabla\cdot B=0$ and
\begin{equation} \label{B-decay-cond}
\abs{B(x)}\lesssim\x^{-\beta-1},\quad x\in\R^3.
\end{equation}
\end{assumption}

In the proof of the following lemma we explain our choice of gauge for $B$ satisfying the above assumption.

\begin{lem}\label{A-choice}
There exists a vector potential $A\colon\R^3\to\R^3$ with $\rt A=B$ such that
\begin{equation}\label{A-decay}
\abs{A(x)}\lesssim\x^{-\beta},\quad x\in\R^3.
\end{equation} 
\end{lem}
\begin{proof}
Let 
\begin{equation} 
A_p(x) = \int_0^1 B(t x)\, t\, dt \wedge x 
\end{equation}
denote the vector potential associated to $B$ via the Poincar\'{e} gauge. Moreover, let 
\begin{equation}  \label{LS-range}
a_\ell(x) = \int_0^\infty B(t x)\, t\, dt \wedge x, \qquad a_s(x) = \int_1^\infty B(t x)\, t\, dt \wedge x
\end{equation}
be the long and the short range components of $A_p$. Note that $a_\ell, a_s\colon\R^3\setminus\{0\}\to\R^3$, and that $A_p= a_\ell -a_s$. The crucial observation is 
that since $B$ is a magnetic field, we have $\nabla \cdot B=0$ , and a short calculation gives $\nabla \wedge a_\ell =0$ in $\R^3\setminus\{0\}$. 
Since $\R^3\setminus\{0\}$ is simply connected, there exists $\widetilde\varphi \in C^2(\R^3\setminus\{0\})$ such that $\nabla \widetilde\varphi = a_\ell$. Note however that 
\begin{equation*}
\abs{a_\ell(x)} \sim \abs{x}^{-1} \quad 
\text{as $\abs{x}\to 0$},
\end{equation*}
by scaling. Hence in order to construct a vector potential $A$ which satisfies \eqref{A-decay} we have to modify $\widetilde\varphi$ in the vicinity of the origin. By Tietze's extension theorem there exists 
$\varphi \in C^2(\R^3)$ such that $\varphi(x)=\widetilde\varphi(x)$ for all $x$ with $\abs{x} \geq 1$. Now we define $A\colon\R^3\to\R^3$ by 
\begin{equation} \label{A-gauge-defin}
A =A_p -\nabla \varphi .
\end{equation}
Then $A\in C^1(\R^3;\R^3)$ and for all $\abs{x}\geq 1$ we have 
\begin{align*}
\abs{A(x)} &\leq \abs{x}  \int_1^\infty t \abs{B(tx)} dt  =  \abs{x}^{-1} \int_{\abs{x}}^\infty s \abs{B(s\abs{x}^{-1}x)}ds 
\\
&\leq C \abs{x}^{-1} \int_{\abs{x}}^\infty   \langle s \rangle^{-\beta} ds \leq C  \langle x \rangle^{-\beta},
\end{align*}
as required.
\end{proof} 

\begin{rem}
The fact that for a given short range magnetic field in $\R^3$ it is always possible to construct a short range vector potential $A$, contrary to the case of dimension two, is well-known, cf.~\cite{ya}.
\end{rem}

 We consider the operator $H$ as a perturbation of $-\Delta$, denoted by $W$, i.e. we define
\begin{equation}\label{W-def}
W = H + \Delta= -P\cdot A -A\cdot P +\abs{A}^2 +V.
\end{equation}
Note that $W$ is a first order differential operator and thus a local operator.
The following lemma is stated without proof.
\begin{lem}\label{compact}
Let $B$ and $V$ satisfy Assumption~\ref{ass-BV} and let $A$ be chosen as in Lemma~\ref{A-choice}. Then  $W$ is a compact operator from $H^{1,s}$ to $H^{-1,s+\beta'}$ for any $s\in\R$ and $\beta'<\beta$.
\end{lem}

\section{Properties of the free resolvent}
\label{sec-free-exp}
Let  $R_0(z)=(-\Delta-zI)^{-1}$, $z\in\C\setminus[0,\infty)$. We recall some  properties of this resolvent from \cite{JK,jn}. We use the conventions from \cite{jn}.
For $z\in\C\setminus[0,\infty)$ let $\kappa=-i\sqrt{z}$, where 
$\im \sqrt{z}>0$, such that $z=-\kappa^2$. We write $R_0(\kappa)$ instead of $R_0(-\kappa^2)$ in the sequel.

\begin{lem}[{\cite[Lemma 2.2]{JK}}]\label{free-exp}
Assume $p\in\N_0$ and $s>p+\frac32$. Then 
\begin{equation}
R_0(\kappa)=\sum_{j=0}^p\kappa^jG_j + \cO(\kappa^{p+1})
\end{equation}
as $\kappa\to0$, $\re\kappa>0$, in $\B(-1,s;1,-s)$. Here
the operators $G_j$ are given by their integral kernels
\begin{equation}
G_j(x,y)=(-1)^j\frac{\abs{x-y}^{j-1}}{4\pi j!},\quad j\geq0.
\end{equation}
We have
\begin{equation}\label{G0-map}
G_0\in\B(-1,s;1,-s')\quad\text{for $s,s'>\tfrac12$ and $s+s'\geq2$}, 
\end{equation}
and for $j\geq1$
\begin{equation}\label{Gj-map}
G_j\in\B(-1,s;1,-s')\quad\text{for $s,s'>j+\tfrac12$}.
\end{equation}
\end{lem}

\section{The factored resolvent equation}
\label{sec-factored}
We will treat the operator $H$ as a perturbation of $-\Delta$.  
Write 
\begin{equation*}
A =(A_1,A_2, A_3) \quad \text{with  $A_j = D_j C_j$}, 
\end{equation*}
where 
\begin{equation} \label{BC-decay}
\abs{D_j(x)}  \lesssim\langle x \rangle^{-\beta/2} ,\quad \abs{C_j(x)}  \lesssim  \langle x\rangle^{-\beta/2}.
\end{equation}
Now let 
\begin{equation}\label{K-def}
\K= L^2(\R^3)\oplus L^2(\R^3;\C^3) \oplus L^2(\R^3;\C^3) \oplus L^2(\R^3;\C^3) ,
\end{equation} 
and put
\begin{align*}
v(x) &= \sqrt{\abs{V(x)}}, 
\\
U(x) &=\begin{cases}
-1,   &\text{if  $V(x) <0$}, \\
\phantom{-}1,  &\text{otherwise}.
\end{cases}  
\end{align*}
We define an operator matrix 
 by
\begin{equation} \label{omega}
w = \begin{bmatrix} v & A_1 & A_2 & A_3 & C_1 & C_2 & 
C_3 & D_1 P_1 & D_2 P_2 & D_3 P_3
\end{bmatrix}^T. 
\end{equation}
Under Assumption~\ref{ass-BV} and the choice~\eqref{BC-decay} we have
\begin{equation}\label{w-map}
w\in\B(H^{1,-s},\K)\quad \text{and}
\quad w^*\in\B(\K,H^{-1,s})\quad \text{for $s\leq\beta/2$}.
\end{equation}
Moreover, we define the block operator matrix $\U\colon\K\to\K$ by
\begin{equation} 
\U =
\begin{bmatrix}
U  & 0 &0  &0 \\
0 & \idd_3 &0 & 0 \\
0 & 0 &0 & -\idd_3 \\
0 & 0 & -\idd_3 & 0
\end{bmatrix}.
\end{equation} 
Here $\idd_3$ denotes the $3\times 3$ unit matrix. Note that $\U$ is self-adjoint and that $\U^2=\idd_{\K}$.
The perturbation $W$ given by \eqref{W-def} then
satisfies
\begin{equation} \label{factorisation}
W = w^* \U w.
\end{equation}

\begin{rem} \label{rem-first-order}
The same factorization method as above can be applied to any self-adjoint first order differential operator perturbation of $-\Delta$ of the form 
$$
 i (L\cdot \nabla +\nabla \cdot L) +V,
$$
as long as the vector field $L\colon \R^3\to \R^3$ is sufficiently regular. Factorization of a more general class of perturbations is discussed in Section~\ref{sec-general}.
\end{rem} 

\noindent To continue we define the operator 
\begin{equation}
M(\kappa)=\U+w R_0(\kappa)w^\ast 
\end{equation}
on $\K$. 

\begin{rem}
Note that for $-\kappa^2\notin\sigma(H)$ the operator $M(\kappa)$ is invertible. This follows from the relation
\begin{equation*}
M(\kappa)\bigl(\U-\U w(H+\kappa^2)^{-1}w^{\ast}\U\bigr)
=\bigl(\U-\U w(H+\kappa^2)^{-1}w^{\ast}\U\bigr)M(\kappa)=I,
\end{equation*}
which is an immediate consequence of the second resolvent equation.
\end{rem}

Lemma~\ref{free-exp} leads to the following result.
\begin{lem}\label{lemma43}
Let $p\in\N$. Assume $\beta>2p+3$. Then
\begin{equation}\label{M-exp}
M(\kappa)=\sum_{j=0}^p\kappa^jM_j+\cO(\kappa^{p+1})
\end{equation}
as $\kappa\to0$, $\re\kappa>0$, in $\B(\K)$.
Here
\begin{equation}\label{M0-3d}
M_0=\U+wG_0w^\ast
\end{equation}
and
\begin{equation}\label{Mj-eq}
 M_j=wG_jw^\ast, \quad j\geq1.
\end{equation}
\end{lem}
For all $-\kappa^2\notin\sigma(H)$  we have the factored resolvent equation
\begin{equation}\label{res-formula-2}
R(\kappa)=R_0(\kappa)-R_0(\kappa)w^{\ast}M(\kappa)^{-1}wR_0(\kappa),
\end{equation}
see e.g.~\cite{jn}.

It follows from \eqref{M-exp} that the operator 
\begin{equation}
\widetilde M_1(\kappa) =\frac 1\kappa \big(M(\kappa)-M_0\big)
\end{equation}
is uniformly bounded as $\kappa\to 0$.
The following inversion formula is needed for the expansion of $M(\kappa)^{-1}$ as $\kappa\to 0$. 
We state it in a form simplified to our setting.  For its general form we refer to \cite{jn,jn2}.

\begin{lem}[{\cite[Corollary~2.2]{jn}}]\label{lem-jn}
Let $M(\kappa)$ be as above. Suppose that $0$ is an isolated point of the spectrum of $M_0$,  and let $S$ be the 
corresponding Riesz projection. 
Then for sufficiently small $\kappa$ the operator $Q(\kappa)\colon S\K\to S\K$ defined by 
\begin{align*}
Q(\kappa) & = \frac 1\kappa \big(S-S(M(\kappa) +S)^{-1} S\big) = \sum_{j=0}^\infty (-\kappa)^j S \big[ \widetilde M_1(\kappa) (M_0+S)^{-1} \big]^{j+1} S
\end{align*}
is uniformly bounded as $\kappa\to 0$. Moreover, the operator $M(\kappa)$ has a bounded inverse in $\K$ if and only if $Q(\kappa)$ has a bounded inverse in
$S\K$ and in this case
\begin{equation}
M(\kappa)^{-1} = (M(\kappa) +S)^{-1} +\frac 1\kappa(M(\kappa) +S)^{-1} S Q(\kappa)^{-1} S (M(\kappa) +S)^{-1}
\end{equation}
\end{lem}

 Proposition \ref{prop-M0} below implies that the hypotheses of Lemma \ref{lem-jn} are satisfied. 

 In view of equation \eqref{res-formula-2} 
the first step in obtaining an asymptotic expansion of $R(\kappa)$ as $\kappa\to0$ consists in analyzing $\ker M_0$. In the sequel we always assume at least $\beta>2$. Under this condition Lemma~\ref{free-exp} implies that $G_0W\in\B(1,-s;1,-s)$ and $WG_0\in\B(-1,s;-1,s)$, provided $\frac12<s<\beta-\frac12$. We define
\begin{align}
M&\coloneqq\{u\in H^{1,-s}\mid (1+G_0W)u=0\},\\
N&\coloneqq\{u\in H^{-1,s}\mid (1+WG_0)u=0\}.
\end{align}
It is shown in \cite{JK} that these spaces are independent of $s$ provided
$\frac12<s<\beta-\frac12$. Furthermore, since $G_0W$ and $WG_0$ are compact (see Lemma~\ref{compact}) we get by duality
\begin{equation}\label{NM-dim}
\dim M= \dim N.
\end{equation}

\noindent We need the following result from \cite{JK}.
\begin{lem}[{\cite[Lemma 2.4]{JK}}]\label{lem-jk}
\
\begin{itemize}
\item[\textup{(1)}] $-\Delta G_0\, u = u$ for any $u\in H^{-1, \frac 12+0}$.  

\item[\textup{(2)}] $G_0(-\Delta) u'=  u'$ for any $u'\in  H^{0, -\frac 32}$ such that $\Delta  u' \in H^{-1, \frac 12+0} $.
\end{itemize}
\end{lem}

The spaces $M$ and $\ker(M_0)$ are related to a generalized null space of $H$ which we define by
\begin{equation}
\nul(H)=\{u\in H^{1,-\frac12-0}\mid Hu=0\},
\end{equation}
where $Hu$ is understood to be in the sense of distributions.

\begin{lem} \label{lem-ker-M0}
Let  Assumption~\ref{ass-BV} be satisfied for some $\beta > 3$.
\begin{itemize}
\item[\textup{(1)}] Let $f\in \ker(M_0)$, and define $u =-G_0w^* f$. Then $u\in M$, and $u\in \nul(H)$.
\item[\textup{(2)}]
Let $u\in M$. Then $u\in \nul(H)$, and $f= \U w u$ satisfies  $f\in 
\ker(M_0)$.
\end{itemize}
\end{lem} 

\begin{proof} 
To prove part (1), assume $f\in\ker M_0$, i.e.~$(\U+wG_0w^*)f=0$. 
Define $u=-G_0w^*f$. 
Since $w^*f\in H^{-1,\beta/2}\subset H^{-1,\frac32+0}$, 
we have $u\in H^{1,-\frac12-0}$ by \eqref{G0-map}.
Lemma~\ref{lem-jk}(1) implies $H_0G_0w^*f=w^*f$ or
$H_0u=-w^*f=w^*\U wG_0w^*f=-Wu$. Thus $u\in\nul(H)$. 
To prove $u\in M$, note that $f\in\ker M_0$ implies $f=\U\U f
=-\U wG_0w^*f=\U wu$. Hence $u=-G_0w^*f=-G_0w^*\U w u=-G_0Wu$, and
$u\in M$ follows.

To prove part (2), let $u\in M$. Then $u\in H^{1,-\frac12-0}$ and
$Wu\in H^{-1,\beta-\frac12-0}\subset H^{-1,\frac12 + 0}$. Lemma~\ref{lem-jk}(1) implies $H_0u=-H_0G_0Wu=-Wu$ and $u\in\nul(H)$ follows. Let $f=\U wu$. Then $f=-\U wG_0Wu=-\U w G_0 w^*\U wu
=-\U w G_0 w^*f$, such that $f\in\ker M_0$ follows.
\end{proof}

Next we define the operators $T_1\colon \ker(M_0) \to M$ and $T_2\colon M \to  \ker(M_0)$ by
\begin{equation} \label{T-12}
T_1 = -G_0 w^*\big|_{\ker(M_0) } \quad 
\text{and} \quad T_2 =  \U w\big|_M.
\end{equation}

\smallskip

\begin{prop} \label{prop-M0}
We have 
\begin{equation} \label{M0-dim}
\dim \ker(M_0) <\infty.
\end{equation}
Moreover, $0$ is an isolated point of $\sigma(M_0)$.
\end{prop}

\begin{proof}
From \eqref{T-12} we get $T_1 T_2 = -G_0 w^* \U w=-G_0 W$, which is the identity operator on $M$. On the other hand $T_2 T_1 = -\U w G_0 w^*$ is the identity operator on $\ker(M_0)$. 
Hence, in view of Lemma~\ref{lem-ker-M0} and~\eqref{NM-dim} we have $\dim\ker(M_0) = \dim M = \dim N < \infty$. 

To prove the second part of the claim we argue by contradiction. Suppose that $0\in\sigma_{\textup{ess}}(M_0)$. Then there exists an orthonormal Weyl sequence $\{u_n\}$ in $\K$ such that 
\begin{equation} \label{un-seq}
\norm{M_0 u_n}_\K \to 0 \quad \text{as $n\to \infty$}. 
\end{equation} 
In particular,  $\{u_n\}$ converges weakly to $0$ in $\K$. 
Let $X= \U w G_0 w^*$. Since $G_0 W$ is compact on $H^{1, -s}$, $\frac12<s<\beta-\frac12$, it follows that the operator 
\begin{equation*}
X^2 = \U w G_0 W G_0 w^* 
\end{equation*}
is compact on $\K$. Hence $X^2 u_n \to 0$  in $\K$. Since $(1+X) u_n = \U M_0 u_n\to 0$ in $\K$ as well (see \eqref{un-seq}) we deduce that 
\begin{equation*}
X u_n= X(1+X) u_n -X^2 u_n \to 0 \qquad \text{in} \quad \K,
\end{equation*}
which implies that $\norm{u_n}_\K \to 0$. However, this is in contradiction with the fact that the sequence $\{u_n\}$ is orthonormal in $\K$.
\end{proof}

\begin{lem}\label{lem-orthogonal} 
Let $u\in \nul(H)$. Then
\begin{equation} \label{L2-condition}
u \in L^2(\R^3) \quad \Leftrightarrow\quad  \langle  u, W 1 \rangle=0.
\end{equation}
\end{lem}

\begin{proof}
If $u\in \nul(H)$, then $u\in H^{1, -\frac 12-0}$ by definition, and therefore $W u \in H^{-1, s}$ for any $s <\min \{\beta  -\frac 12, \frac 52\}$.  Lemma~\ref{lem-jk}(2) then says that $u  = -G_0 W u$. 

Now assume that $\langle  u, W 1 \rangle=0$. Then by \cite[Lemma~2.5]{JK} we have $u=-G_0 Wu \in  H^{1, s-2}$. Hence $W u\in  H^{-1, s +\delta}$ with $\delta =\beta -2>0$. Repeating this
argument a sufficient number of times, we conclude that $W u\in  H^{-1, \frac 52-0}$, and therefore $u\in  H^{1, \frac 12-0}$. 

To prove the opposite implication, suppose that  $u \in L^2(\R^3)$. Then $u_1= \Delta  u= Wu \in H^{-1, \frac 32+0}$, which implies that $(1-\Delta)^{-\frac 12} u_1 \in L^1(\R^3)$. 
Hence $(1+\abs{\,\cdot\,}^2)^{-\frac 12}\,  \widehat{u}_1$ is continuous, 
and therefore so is $\widehat{u}_1$. Since $\widehat{u}_1(p)= -\abs{p}^2 \hat u(p)$ and $\hat u\in L^2(\R^3)$, we must have $\widehat{u}_1(0)=0$. This gives $\ip{u}{W 1}=0$. 
\end{proof}

Next we need to classify the point $0$ in the spectrum of $H$. The classification is the same as in~\cite{JK,jn1}. We recall it for completeness. 

Let $S$ denote the orthogonal projection onto $\ker M_0$ in $\K$, cf.~Lemma \ref{lem-jn}, and let $S_1$ denote the orthogonal projection on $\ker SM_1S$ in $\K$. By Proposition~\ref{prop-M0} $\ker M_0$ is finite dimensional, and by the definition of $M_1$ (see \eqref{Mj-eq}) we have
\begin{equation} \label{sm1s}
SM_1S=-\frac{1}{4\pi}\ket{Sw1}\bra{Sw1}.
\end{equation}
It follows that $\rank S_1\geq \rank S -1$. Note that $f\in\ker SM_1S$ if and only if $\ip{f}{Sw1}=0$.

\smallskip

The classification is then as follows (cf.~\cite{jn1}):

\begin{enumerate}
\item[(R)] The regular case: $S=0$. In this case $M(\kappa)$ is invertible.
\item[(E1)] The first exceptional case: $\rank S=1$ and $S_1=0$. In this case we have a threshold resonance.
\item[(E2)] The second exceptional case: $\rank S=\rank S_1\geq1$. In this case zero is an eigenvalue of multiplicity $\rank S$.
\item[(E3)] The third exceptional case: $\rank S\geq 2$, $\rank S_1=\rank S-1$.
In this case we have a threshold resonance and zero is an eigenvalue with multiplicity $\rank S - 1$.
\end{enumerate}

\section{Main results}
\label{sec-main}
In this section we briefly state the leading terms in the resolvent expansions around zero in the four cases. We start with the regular case and give the proof for completeness. Note that we also give more precise mapping properties than in~\cite{jn3}.

\begin{thm} \label{thm-regular}
Assume that zero is a regular point for $H$. Let 
Assumption~\ref{ass-BV} be satisfied for some 
$\beta > 5$ and let $s>\frac 52$. Then
\begin{equation}  \label{exp-regular}
R(\kappa) = F_0 +\kappa F_1 +\mathcal{O}(\kappa^{2})
\end{equation}
in $\B(-1,s;1,-s)$, where
\begin{align} 
F_0 &= (\id +G_0 W)^{-1} G_0\in\B(-1,s;1,-s),\quad s>1,
\label{RF0}
\\
F_1 &= (\id +G_0 W)^{-1} G_1(\id + W G_0)^{-1}\in\B(-1,s;1,-s),\quad s>\tfrac32.
\label{RF1}
\end{align}
\end{thm}

\begin{proof}
If  $0$ is a regular point for $H$, then $\ker M_0 = \{0\}$. In view of Lemma~\ref{lem-ker-M0} we thus have $\ker(\id +G_0 W)= \{0\}$. 
Since $G_0 W$ is compact in $ H^{1,-s}(\R^3)$ for any $\frac 12 < s <\beta -\frac 12$, it follows that $(\id +G_0 W)^{-1}$
exists and is bounded on $H^{1,-s}(\R^3)$. By duality, $(\id + W G_0)^{-1}$ is  bounded on $H^{-1,s}(\R^3)$ for any $\frac 12 < s <\beta -\frac 12$.
Using \eqref{G0-map}, \eqref{Gj-map}, and \eqref{w-map} the results \eqref{RF0} and \eqref{RF1} follow.

The proof of \eqref{exp-regular} follows the line of arguments used in \cite[Section~3.4]{jn3}. Since $M_0$ is invertible in $\K$ (see Proposition~\ref{prop-M0}), the Neumann series in combination
with equations \eqref{M-exp} and \eqref{Mj-eq} gives 
\begin{equation} 
M(\kappa)^{-1} = M_0^{-1} -\kappa M_0^{-1} M_1 M_0^{-1} + \mathcal{O}(\kappa^{2}) = M_0^{-1} -\kappa M_0^{-1} w G_1 w^* M_0^{-1} + \mathcal{O}(\kappa^{2}) .
\end{equation}
From \eqref{res-formula-2} we the get the expansion
 \eqref{exp-regular} with 
\begin{equation*}
F_0 = G_0 -G_0w^* M_0^{-1} w G_0, \quad F_1= (\id- G_0w^*  M_0^{-1} w) G_1 (\id -w^*  M_0^{-1} w G_0). 
\end{equation*}
It remains to note that, similarly to~\cite[Section~3.4]{jn3},
\begin{align*}
\id- G_0w^*  M_0^{-1} w &= \id -G_0w^* (\U +w G_0 w^*)^{-1} w = \id -G_0w^* \U (\id +w G_0 w^*\U)^{-1} w\\
& = \id -G_0w^* \U w (\id + G_0 w^*\U w)^{-1} = \id -G_0 W(\id+G_0 W)^{-1}\\
&= (\id +G_0 W)^{-1} .
\end{align*}
Note that equalities hold as operators in $\B(1,-s;1,-s)$, $\frac12<s<\beta-\frac12$.
This result together with its adjoint imply equations \eqref{RF0}
and \eqref{RF1}.
\end{proof}

\begin{rem} \label{rem-KK}
The fact that $R(\kappa)$ remains uniformly bounded for $\kappa\to 0$ if zero is a regular point for $H$ was already 
proved in \cite[Sec.~3]{egs}, see also \cite[Sec.~3.2]{kk}. 
\end{rem}

In the cases (E1) and (E3) a threshold resonance occurs. We need to define a specific corresponding resonance function.
Let $P_0$ denote the orthogonal projection in $L^2(\R^3)$ onto the eigenspace corresponding to eigenvalue zero of $H$. In case (E1) we take $P_0=0$. Take $f\in\ker M_0$ with $\norm{f}_{\K}=1$ and $\ip{f}{w1}\neq0$. Define
\begin{equation} \label{psi-c-def}
\psi_c=\frac{\sqrt{4\pi}\ip{f}{w1}}{\abs{\ip{f}{w1}}^2}
\bigl(
G_0w^*f-P_0WG_2w^*f
\bigr).
\end{equation}

We need the following lemma.
\begin{lem}[{\cite[Lemma~2.6]{JK}}]
Let Assumption~\ref{ass-BV} be satisfied with $\beta>5$. Let $f_j\in\K$ with $\ip{f_j}{w1}=0$, $j=1,2$. Then
\begin{equation}
\ip{w^*f_1}{G_2w^*f_2}=-\ip{G_0w^*f_1}{G_0w^*f_2}.
\end{equation}
\end{lem}

The results in the three exceptional cases are stated in the next theorem.

\begin{thm}\label{thm53}
Assume that zero is an exceptional point for $H$. Let Assumption~\ref{ass-BV} be satisfied for $\beta>9$. Assume $s>\frac92$. Then
\begin{equation} \label{exp-exceptional}
R(\kappa)=\kappa^{-2}F_{-2}+\kappa^{-1}F_{-1}+\cO(1)
\end{equation}
as $\kappa\to0$ in $\B(-1,s;1,-s)$.

If zero is an exceptional point of the first kind, we have
\begin{equation}
F_{-2}=0,\quad F_{-1}=\ket{\psi_c}\bra{\psi_c}.
\end{equation} 

If zero is an exceptional point of the second kind, we have
\begin{equation}
F_{-2}=P_0,\quad F_{-1}=P_0WG_3WP_0.
\end{equation}

If zero is an exceptional point of the third kind, we have
\begin{equation}
F_{-2}=P_0,\quad F_{-1}=\ket{\psi_c}\bra{\psi_c}+P_0WG_3WP_0.
\end{equation}
\end{thm}
We do not give details of the proof of this theorem. It uses the results stated in Section \ref{sec-factored} and the technique developed in \cite{jn,jn2} and is analogous to the one given in \cite[Appendix]{jn1} and in \cite{jn3}.

\begin{rem}[Gauge invariance]The resolvent expansions stated in Theorems \ref{thm-regular} and \ref{thm53} hold for the specific choice of the vector potential constructed in Lemma \ref{A-choice}. Therefore a comment on the gauge dependence of these results  is in order.  Suppose that $\widetilde A \in  C^1(\R^3;\R^3)$ satisfies $\rt \widetilde A=B$. Then there exists a real-valued function $\vp\in C^1(\R^3)$ with bounded derivatives  such that 
$\widetilde{A}=A-\nabla\vp$, with $A$ given by Lemma \ref{A-choice}. 
Now let
$\cG=e^{-i\vp}\cdot$, $\cG^*=e^{i\vp}\cdot$. Then 
\begin{equation*}
\cG\in\B(1,-s;1,-s),\quad \cG^*\in\B(-1,s;-1,s),
\end{equation*}
and 
$$
\widetilde{R}(\kappa)= ((P-\widetilde{A})^2+V +\kappa^2)^{-1} = \cG R(\kappa) \cG^*.
$$
Hence $\widetilde{R}(\kappa)$ satisfies expansions \eqref{exp-regular} respectively \eqref{exp-exceptional} with coefficients  $F_j$ replaced by $\cG F_j\cG^*$. In other words, the order of magnitude of 
the terms contributing to the expansion is gauge invariant, but the coefficients are not. This is natural since the resolvent itself is not gauge invariant. 

 \end{rem}
\subsection{The case $V\geq 0.$}
\label{ssec-mg-laplace} 
The goal of this subsection is to show that if  $V\geq 0$, then zero is a regular point for $H$. We will need slightly stronger conditions on $B$ than those stated in Assumption \ref{ass-BV}.

\begin{assumption}\label{ass-B-2}
Let $B$ satisfy Assumption \ref{ass-BV} and suppose in addition that 
\begin{equation}
\abs{\partial_{x_j} B(x)}\lesssim   \x^{-\beta-1}
\end{equation}
for all $x\in\R^3$, $j=1,2,3$.
\end{assumption}

 We start with the magnetic Laplacian.  

\begin{lem} \label{lem-mg-laplacian}
Let $V=0$ and let $B$ satisfy Assumption \ref{ass-B-2} for some $\beta >2$. Then $\ker M_0 = \{0\}$. 
\end{lem}

\begin{proof}
Owing to Lemma \ref{lem-ker-M0} it suffices to show that $\nul((P-A)^2)=\{0\}$. So let  $u\in H^{1,-\frac12-0}$ be such that $(i\nabla +A)^2u=-\Delta u + Wu =0$. 
We have
\begin{equation}\label{W-bis}
W = 2 i A \cdot\nabla + i\Div A  +\abs{A}^2 .
\end{equation}
By equations \eqref{LS-range}, \eqref{A-gauge-defin} and Assumption \ref{ass-B-2},  for all $\abs{x}\geq 1$ it holds 
\begin{equation}
\abs{\Div A(x)} =   \abs{\Div a_s(x)} = \Bigl\lvert  x \cdot \Bigl (\int_1^\infty \nabla \wedge B(tx) \, t\, dt\Bigr) \Bigr\rvert \leq C  \langle x \rangle^{-\beta}. 
\end{equation}
Hence from Lemma \ref{A-choice} and equation \eqref{W-bis} we deduce that $Wu\in H^{0, \beta -\frac 12-0}(\R^3)$, and therefore, by H\"older, $Wu\in L^{\frac 65}(\R^3)$. 
Moreover, $Wu \in L^2_{\textrm{loc}}(\R^3)$, so by the elliptic regularity we have $u\in H^2_{\textrm{loc}}(\R^3)$.
By Lemma \ref{lem-ker-M0}, $u=- G_0Wu$, hence
\begin{equation} 
u(x) = -\frac{1}{4\pi} \int_{\R^3} \frac{(Wu)(y)}{\abs{x-y}}\, dy .
\end{equation}
In view of the regularity of $u$ we have $Wu \in H^1_{\textrm{loc}}(\R^3)$, and therefore $Wu \in L^6_{\textrm{loc}}(\R^3)$. Thus
\begin{equation} 
\abs{\partial_{x_j} u(x)} \leq \frac{1}{4\pi} \int_{\R^3} \frac{\abs{(Wu)(y)}}{\abs{x-y}^2}\, dy .
\end{equation}
Since $Wu\in L^{\frac 65}(\R^3)$, the Hardy-Littlewood-Sobolev inequality, see e.g.~\cite[Section~4.3]{LL}, then implies, by duality, that $\abs{\nabla u}\in L^2(\R^3)$, and therefore $\abs{(i \nabla +A) u} \in L^2(\R^3)$. Now let 
$\chi_n\colon \R^3\to\R$, $2\leq n\in\N,$ be given by 
$$
\chi_n(x)= 1 \quad \text{if $\abs{x} \leq 1$}, \qquad \chi_n(x)= \Bigl(1- \frac{\log \abs{x}}{\log n}\Bigr)_+ \quad \text{otherwise}.
$$
Then
\begin{align} \label{zero-eq}
0&= \int_{\R^3} \chi_n \bar u\, (i \nabla +A)^2 u \, dx  
\notag\\
&= - \int_{\R^3} \chi_n \abs{(i \nabla +A) u}^2 \, dx - \int_{\R^3}\bar u\,  \nabla \chi_n \cdot (i \nabla +A) u \, dx.
\end{align}
A short calculation gives
$$
\nabla \chi_n(x)= -\frac{x}{\abs{x}^2 \log n}\ \quad \text{if $1\leq \abs{x} \leq n$}, \qquad  \nabla \chi_n(x)= 0 \quad \text{otherwise},
$$
we get, for any $0<\eps \leq \frac 12$,
\begin{align*}
\Bigl\lvert\int_{\R^3}\bar u\,  \nabla \chi_n \cdot (i \nabla +A) u \, dx\Bigr\rvert &\leq   \frac{1}{\log n}\int_{1\leq \abs{x}\leq n} \abs{u}  \langle x\rangle^{-\frac 12-\eps}  \frac{\langle x\rangle^{\frac 12+\eps}}{\abs{x}} \abs{(i \nabla +A) u }\, dx\\[4pt]
& \leq  \frac{1}{\log n}\,  \norm{u}_{L^{2,-\frac 12-\eps}} \norm{(i \nabla +A) u}_{L^2} \ \to \ 0 \quad \text{as $n\to \infty$} .
\end{align*}
Since $\chi_n \to 1$ in $L^\infty_{\textrm{loc}}(\R^3)$, this in combination with equation \eqref{zero-eq} gives 
\begin{equation} \label{diamag-eq}
(i \nabla +A) u = 0.
\end{equation}
However, \eqref{diamag-eq} implies $\abs{u} = \textrm{const}$, which in view of $u\in H^{1,-\frac12-0}$ means that $u=0$, see also \cite{si}.
\end{proof}

\begin{cor} \label{cor-regular}
Suppose that $V$ and $B$ satisfy Assumptions \ref{ass-BV} respectively \ref{ass-B-2} for some $\beta >2$. If $V\geq 0$, then $\ker M_0 = \{0\}$. 
\end{cor}

\begin{proof}
Let $u\in \nul \bigl((P-A)^2 +V\bigr)$.  Following the arguments of the proof of Lemma \ref{lem-mg-laplacian} we deduce that $u$ must satisfy 
\begin{equation} 
\int_{\R^3} \abs{(i \nabla +A) u}^2 \, dx + \int_{\R^3} V\, \abs{u}^2\, dx =0.
\end{equation}
Note that  $\sqrt{V} u\in L^2(\R^3)$, by hypothesis. 
Since $V\geq 0$, we conclude, as above, that  $(i \nabla +A) u = 0$ and therefore $u=0$.
\end{proof}

\section{The Pauli operator}
\label{sec-pauli}

We assume that $B\colon\R^3\to \R^3$  and $V\colon\R^3\to \R$ satisfy Assumption \ref{ass-BV}. 
In what follows we denote by $\idd_n$ the $n\times n$ identity matrix.
We consider the Pauli operator in $L^2(\R^3;\C^2)$ given by 
\begin{equation} 
H_P=  \bigl(\sigma\cdot (P-A)\bigr)^2  + V\idd_2 =  (P-A)^2 \idd_2 + \sigma\cdot B + V \idd_2,
\end{equation} 
where $\sigma=(\sigma_1,\sigma_2,\sigma_3)$ is the set of Pauli matrices;
\begin{equation}\label{sigma-j}
\sigma_1 =
\begin{bmatrix}
0  & 1 \\
1 & 0
\end{bmatrix},
\quad
\sigma_2 =
\begin{bmatrix}
0  & -i \\
i & 0
\end{bmatrix},
\quad
\sigma_3 =
\begin{bmatrix}
1  & 0 \\
0 & -1
\end{bmatrix},
\end{equation}
and where $A$ is given by Lemma \ref{A-choice}. Here we adopt the usual notation
\begin{equation} \label{sigma-dot-B}
\sigma\cdot B = \sum_{j=1}^3 \sigma_j B_j .
\end{equation}

\noindent Hence
 \begin{equation} \label{pauli-def-2}
H_P=   -\Delta\idd_2 +W_P,
\end{equation} 
where
\begin{align} \label{W-pauli}
W_P &= W_1+W_2,\\
W_1&=\bigl(-P\cdot A -A\cdot P +\abs{A}^2\bigr) \idd_2
+ V\idd_2,
\label{W1-def}\\
W_2&= \sigma\cdot B.
\end{align}
Our aim is to factor the perturbation $W_P$ in a way similar to  the scalar case. To this end note that $W_1=W\idd_2$ with $W$ defined in \eqref{W-def}. We take as the intermediate space
\begin{equation}
\K_1=\K\oplus\K,
\end{equation}
where the space $\K$ is defined in \eqref{K-def}. Thus the factorization of $W$ given in \eqref{factorisation} immediately gives a factorization of $W_1$. We can write it as $W_1={w_1}^{\!\!\ast\,}\U_1w_1$,  where
$w_1=w\oplus w$ and $\U_1=\U\oplus\U$.

To factor $W_2$ we take as the intermediate space
\begin{equation}
\K_2=L^2(\R^3;\C^2)\oplus L^2(\R^3;\C^2)\oplus L^2(\R^3;\C^2).
\end{equation}
Then let
\begin{equation}
U^B_j(x) =
\begin{cases}
-1,   &\text{if $B_j(x) <0$}, \\
\phantom{-}1,  & \text{otherwise},
\end{cases}
\end{equation}
for $j=1,2,3$.
Define the block-diagonal matrix operator
\begin{equation}
\U_2=\diag
\begin{bmatrix}
U^B_1\sigma_1 & U^B_2\sigma_2 & U^B_3\sigma_3
\end{bmatrix}
\end{equation}
Let $b_j(x)=\abs{B_j(x)}^{\frac12}$ and then define the block-matrix operator
\begin{equation} 
w_2=\begin{bmatrix}
b_1\idd_2 & b_2\idd_2 & b_3\idd_2
\end{bmatrix}^T
\end{equation}
Then we have the factorization $W_2={w_2}^{\!\!\ast\,}\U_2w_2$. We can now put the two factorizations together. The intermediate space is
\begin{equation}
\K_P=\K_1\oplus\K_2,
\end{equation}
and the factorization 
\begin{equation}\label{W-factorization-pauli}
W_P={w_P}^{\!\!\ast\,}\U_Pw_P
\end{equation} 
 is obtained by taking
\begin{equation}
w_P=w_1\oplus w_2\quad\text{and}\quad \U_P=\U_1\oplus\U_2.
\end{equation}

The operators $G_j$ introduced in Section~\ref{sec-free-exp} act as matrix diagonal operators $G_j\idd_2$ from $H^{-1,s}(\R^3;\C^2)$ to $H^{1,-s'}(\R^3;\C^2)$. For simplicity we continue to use the notation 
$\B(-1,s;1,-s')$ for bounded operators between these spaces. The inner product
in the space $L^2(\R^3; \C^2)$ is still denoted by
$\ip{\cdot}{\cdot}$.

With equation \eqref{W-factorization-pauli} at hand we 
can thus write the resolvent of the Pauli operator 
\begin{equation} 
R_P(\kappa) = (H_P+\kappa^2)^{-1}
\end{equation}
in the factorized form as in \eqref{res-formula-2}. By carrying over the analysis of Section~\ref{sec-factored}  to the setting of operators defined on
$L^2(\R^3;\C^2)$ 
we obtain the `matrix versions' of Lemmas~\ref{lem-ker-M0}, \ref{lem-orthogonal} and 
of Proposition~\ref{prop-M0}. As for Lemma \ref{lem-orthogonal}, it takes the following form.

\begin{lem}\label{lem-orthogonal-pauli} 
If $u\in  \nul(H_P)$, then
\begin{equation*} 
u \in L^2(\R^3;\C^2) \quad \Leftrightarrow\quad  
\ip{W_Pu}{\begin{bsmallmatrix} 1 \\ 1 \end{bsmallmatrix}}=0.
\end{equation*}
\end{lem}

\begin{proof}
If $u\in \nul(H_P)$, let
\begin{equation*}
u= \begin{bmatrix}
 u_1 \\
 u_2
\end{bmatrix}\quad\text{and}\quad
W_Pu=\w= \begin{bmatrix}
 \w_1 \\
 \w_2
\end{bmatrix}.
\end{equation*}
 Then, as in Lemma \ref{lem-orthogonal}, we conclude that  $\w_j \in H^{-1, s}$ for any $s <\min \{\beta  -\frac 12, \frac 52\}$, $j=1,2$.

Now assume that $\langle  1, \w_j \rangle_{L^2(\R^3)}=0$. Then, again following the proof of Lemma \ref{lem-orthogonal},  we conclude that $\w_j\in  H^{-1, \frac 52-0}$. Since $u_j = -G_0\, \w_j$, this implies $u_j\in  H^{1, \frac 12-0} \subset L^2(\R^3)$, see \eqref{G0-map}. 

To prove the opposite implication, suppose that  $u_j \in L^2(\R^3), j =1,2$. Then $\Delta  u_j= \w_j\in H^{-1, \frac 32+0}$, which implies that $(1-\Delta)^{-\frac 12}\, \w_j \in L^1(\R^3)$. 
Hence $(1+\abs{\,\cdot\,}^2)^{-\frac 12}\,  \widehat{\w}_j$ is continuous, 
and therefore so is $\widehat{\w}_j$. Since $\widehat{\w}_j(p)= -\abs{p}^2\, \widehat u_j(p)$ and $\widehat u_j\in L^2(\R^3)$, we must have $\widehat{\w}_j(0)=0$. This gives $\langle  1, \w_j \rangle_{L^2(\R^3)}=0$. 
\end{proof}

In order to give a classification of the point zero analogous to the one in 
Section~\ref{sec-factored} we need to go back to the expansion of $M(\kappa)$ in the Pauli case, cf. Lemma~\ref{lemma43}. Explicitly, we use the following notation:
\begin{equation}
M_P(\kappa)=\sum_{j=0}^p\kappa^jM_{P,j}+\cO(\kappa^{p+1}),
\end{equation}
where
\begin{equation}
M_{P,0}=\U_P+w_PG_0\idd_2w_P^{\ast}
\end{equation}
and
\begin{equation}
 M_{P,j}=w_PG_j\idd_2w_P^\ast, \quad j\geq1.
\end{equation}

The detailed structure of the term $M_{P,1}$ is needed in the sequel. It is an operator on the intermediate space $\K_P$ which can be identified with $L^2(\R^3;\C^{26})$. We can consider $w_P$ to be a map from $H^1(\R^3;\C^2)$ to $L^2(\R^3;\C^{26})$, which can be represented as a $26\times2$ operator matrix. We will write it in block form as $w_P=[\alpha,\beta]$, where $\alpha$ and $\beta$ are $26\times1$ operator matrices. Introduce the following notation for a decomposition of $\{1,2,\ldots,26\}$ into two disjoint indexing sets
\begin{equation*}
J_1=\{11,12,\ldots,20,22,24,26\}\quad\text{and}\quad
J_2=\{1,2,\ldots,10,21,23,25\}.
\end{equation*}
Then we define
\begin{equation}\label{alpha-beta}
\alpha=\begin{cases}
v, & j=1,\\
A_{j-1}, & j=2,3,4,\\
C_{j-4}, & j=5,6,7,\\
D_{j-7}P_{j-7}, & j=8,9,10,\\
b_{j-20}, & j=21,\\
b_{j-21}, & j=23,\\
b_{j-22}, & j=25,\\
0, & j\in J_1,
\end{cases}
\quad\text{and}\quad
\beta=\begin{cases}
v, & j=11,\\
A_{j-11}, & j=12,13,14,\\
C_{j-14}, & j=15,16,17,\\
D_{j-17}P_{j-17}, & j=18,19,20,\\
b_{j-21}, & j=22,\\
b_{j-22}, & j=24,\\
b_{j-23}, & j=26,\\
0, & j\in J_2.
\end{cases}
\end{equation}
With these definitions we can compute an expression for $M_{P,1}$, viz.
\begin{equation}\label{MP1}
M_{P,1}=-\frac{1}{4\pi}\bigl(\ket{\alpha1}\bra{\alpha1}+\ket{\beta1}\bra{\beta1}\bigr).
\end{equation}
Note that due to the assumptions on $V$ and $B$, and the choice of $A$, we have $\alpha1,\beta1\in\K_P$. Furthermore, due to the structure of $\alpha$ and $\beta$ it follows that  $\alpha1$ and $\beta1$ are simultaneously either nonzero or zero. We also have that
$\alpha1$ and $\beta1$ are orthogonal, and  
$\norm{\alpha1}=\norm{\beta1}$. As a consequence of these observations we have either $\rank M_{P,1}=0$ or $\rank M_{P,1}=2$.

Let $S_P$ denote the orthogonal projection on $\ker M_{P,0}$. Due to the 
Pauli operator version of Proposition~\ref{prop-M0} we have 
$\rank S_P<\infty$. Let $S_{P,1}$ denote the orthogonal projection on 
$\ker S_PM_{P,1}S_P$. The results above show that
\begin{equation}
\rank S_{P,1}\geq \rank S_P-2.
\end{equation}

\begin{lem}
Let $f\in\ker M_{P,0}$. Then $f\in\ker S_PM_{P,1}S_P$ if and only if
\begin{equation}\label{oneone}
\ip{f}{w_P\oneone}=0.
\end{equation}
\end{lem}
\begin{proof}
Let $f\in\ker M_{P,0}$ and assume $f\in\ker S_PM_{P,1}S_P$. Then
\begin{align*}
0&=\ip{f}{M_{P,1}f}=-\frac{1}{4\pi}\bigl(
\ip{f}{\alpha1}\ip{\alpha1}{f}+\ip{f}{\beta1}\ip{\beta1}{f}
\bigr)
\\
&=-\frac{1}{4\pi}\bigl(
\abs{\ip{f}{\alpha1}}^2+\abs{\ip{f}{\beta1}}^2
\bigr).
\end{align*}
Thus $\ip{f}{\alpha1}=0$ and $\ip{f}{\beta1}=0$, which can be written as
$\ip{f}{w_P\oneone}=0$. The converse is obvious.
\end{proof}

Rewrite the left hand side of \eqref{oneone} as follows. 
Let $f\in\ker M_{P,0}$ and
define $u=-G_0\idd_2w_P^{\ast}f$. Then by Lemma~\ref{lem-ker-M0} in the Pauli case $u\in\nul H_P$. We have $f=\U_Pw_Pu$. Thus
\begin{equation}
\ip{f}{w_P\oneone}=\ip{\U_Pw_Pu}{w_P\oneone}=\ip{W_Pu}{\oneone}.
\end{equation}

With these preparations we can state the classification of the point zero in the spectrum of $H_P$.
\begin{enumerate}
\item[(R)] The regular case: $S_P=0$. In this case $M(\kappa)$ is invertible.
\item[(E1)] The first exceptional case: $\rank S_P\in\{1,2\}$ and $S_{P,1}=0$. In this case we have a multiplicity one or a multiplicity two threshold resonance.
\item[(E2)] The second exceptional case: $\rank S_P=\rank S_{P,1}\geq1$. In this case zero is an eigenvalue of multiplicity $\rank S_P$.
\item[(E3)] The third exceptional case: (1) if $\rank S_P\geq 2$
and $\rank S_{P,1}=\rank S_P-1$ we have a multiplicity one threshold resonance, and zero is an eigenvalue of multiplicity $\rank S_P - 1$. (2) if $\rank S_P\geq 3$
and $\rank S_{P,1}=\rank S_P-2$ we have a multiplicity two threshold resonance, and zero is an eigenvalue of multiplicity $\rank S_P - 2$.
\end{enumerate}

\begin{thm} \label{thm-regular-pauli}
Assume that zero is a regular point for $H_P$. Let 
Assumption~\ref{ass-BV} be satisfied for some 
$\beta > 5$ and let $s>\frac 52$. Then
\begin{equation}  \label{exp-regular-pauli}
R_P(\kappa) =  F_0 +\kappa F_1 +\mathcal{O}(\kappa^{2})
\end{equation}
in $\B(-1,s;1,-s)$, where
\begin{align} 
F_0 &= (\id +G_0\idd_2 W_P)^{-1} G_0\idd_2\in\B(-1,s;1,-s),\quad s>1,
\label{RF0-pauli}
\\
F_1 &= (\id +G_0\idd_2 W_P)^{-1} G_1\idd_2(\id + W_P
 G_0\idd_2)^{-1}\in\B(-1,s;1,-s),\quad s>\tfrac32.
\label{RF1-pauli}
\end{align}
\end{thm}

For later purposes we state also a simplified version of \eqref{exp-regular-pauli}. 

\begin{cor} \label{cor-regular-pauli}
Assume that zero is a regular point for $H_P$. Let 
Assumption~\ref{ass-BV} be satisfied for some 
$\beta > 3$. Assume $s,s'>\tfrac 12$ and $s+s' \geq 2$. Then $\lim_{\kappa\to0}R_P(\kappa)$ exists in $\B(-1,s;1,-s')$ and
\begin{equation}  
R_P(0) =  (\id +G_0\idd_2 W_P)^{-1} G_0\idd_2\in\B(-1,s;1,-s').
\end{equation}
\end{cor}

\begin{proof}
Note that it suffices to prove the result for $s,s'$ small and satisfying the conditions in the corollary, due to the embedding property~\eqref{embed}.
The claim then follows from the mapping properties of $G_0$, see equation \eqref{G0-map}, and from the fact that $(\id +G_0\idd_2 W_P)^{-1}$
exists and is bounded on $H^{1,-s}(\R^3;\C^2)$ for any $\frac 12 < s <\beta -\frac 12$, see the proof of Theorem~\ref{thm-regular}. 
\end{proof}

In the exceptional cases
the multiplicity of a zero resonance can be either one or two. The multiplicity one case can be handled as in the previous section, and the resonance function $\psi_c$ is given by the Pauli analogue of \eqref{psi-c-def}.
The multiplicity two threshold resonance case can be handled by going through the computations in \cite[Appendix~A]{jn1}. We will give a few of the steps in this procedure. The key point is the analogue of \cite[(A.46)]{jn1}.

We introduce  the analogue of \cite[(A.28)]{jn1}:
\begin{equation}
m_{P,0}=S_PM_{P,1}S_P.
\end{equation}
Using \eqref{MP1} we get
\begin{equation}
m_{P,0}=-\frac{1}{4\pi}\bigl(\ket{S_P\alpha1}\bra{S_P\alpha1}
+\ket{S_P\beta1}\bra{S_P\beta1}\bigr).
\end{equation}
We need to find the inverse of the operator $m_{P,0}+S_{P,1}$ in $S_P\K_P$.
The details are given in Appendix~\ref{appA}.

The results in the exceptional cases can be stated as follows. See Appendix~\ref{appA} for the construction of the resonance functions 
$\psi_c^1,\psi_c^2\in\nul(H_P)$.

\begin{thm} \label{thm-ex-pauli}
Assume that zero is an exceptional point for $H_P$. Let Assumption~\ref{ass-BV} be satisfied for $\beta>9$. Assume $s>\frac92$. Then
\begin{equation}
R_P(\kappa)=\kappa^{-2}F_{-2}+\kappa^{-1}F_{-1}+\cO(1)
\end{equation}
as $\kappa\to0$ in $\B(-1,s;1,-s)$.

If zero is an exceptional point of the first kind, and the threshold resonance has multiplicity one, we have
\begin{equation}
F_{-2}=0,\quad F_{-1}=\ket{\psi_c}\bra{\psi_c}, 
\end{equation}
 where $\psi_c$ is given by  \eqref{psi-c-def} with $G_0$ replaced by $G_0\idd_2$ and with $P_0=0$. 

\noindent In case of multiplicity two we have
\begin{equation}
F_{-2}=0,\quad F_{-1}=\ket{\psi_c^1}\bra{\psi_c^1}
+\ket{\psi_c^2}\bra{\psi_c^2}.
\end{equation}

If zero is an exceptional point of the second kind, we have
\begin{equation}
F_{-2}=P_0,\quad F_{-1}=P_0W_PG_3\idd_2W_PP_0.
\end{equation}

If zero is an exceptional point of the third kind, and the threshold resonance has multiplicity one, we have
\begin{equation}
F_{-2}=P_0,\quad F_{-1}=\ket{\psi_c}\bra{\psi_c}+P_0W_PG_3\idd_2W_PP_0,
\end{equation}
where $\psi_c$ is given by  \eqref{psi-c-def} with $G_0$ and $G_2$ replaced by $G_0\idd_2$ and $G_2\idd_2$. 

\noindent In case of multiplicity two we have
\begin{equation}
F_{-2}=P_0,\quad 
F_{-1}=\ket{\psi_c^1}\bra{\psi_c^1}
+\ket{\psi_c^2}\bra{\psi_c^2}+P_0W_PG_3\idd_2W_PP_0.
\end{equation}
\end{thm}

\subsection{The case $V=0$.} 
In this subsection we will analyze more in detail the purely magnetic Pauli operator
\begin{equation} \label{pauli-def}
H_P=  \bigl(\sigma\cdot (P-A)\bigr)^2 =  (P-A)^2 +\sigma\cdot B.
\end{equation} 
Notice that in view of the assumptions on $A$, the operator $\bigl(\sigma\cdot (P-A)\bigr)^2$ is self-adjoint on $H^2(\R^3;\C^2)$.

It is well-known that, contrary to purely magnetic Schr\"odinger operators, zero might be an exceptional point of $\bigl(\sigma\cdot (P-A)\bigr)^2$, see~\cite{ly,elt}. Our next result shows that, under suitable conditions on $B$, in such a case zero must be an eigenvalue of $\bigl(\sigma\cdot (P-A)\bigr)^2$ and that there is no threshold  resonance. Our proof is based on an analogous, and more general, result for the Dirac operator obtained in  \cite{fl}.

\begin{lem} \label{lem-no-resonance}
Let $V=0$ and let $B$ satisfy Assumption \ref{ass-B-2} for some $\beta >2$. Then $\rank S_P=\rank S_{P,1}$.
\end{lem}

Recall that $S_P$ and $S_{P,1}$ are orthogonal projections onto $\ker M_{P,0}$ and $\ker S_PM_{P,1}S_P$ in $\K_P$, respectively.
\begin{proof}
It suffices to consider the case  $S_P\neq 0$.
 Let $f\in \ker M_{P,0}$. By Lemma \ref{lem-ker-M0}(i), $u=-G_0\idd_2 {w_P}^{\!\ast\,} f\in  \nul(H_P)$. Our goal is to show that 
$u\in L^2(\R^3; \C^2)$. Since $u\in \ker (\id+G_0\idd_2 W_P)$, 
see Lemma \ref{lem-ker-M0}(i), we have
\begin{equation} 
u(x) = -\frac{1}{4\pi} \int_{\R^3} \frac{(W_Pu)(y)}{\abs{x-y}}\, dy .
\end{equation}
The Hardy-Littlewood-Sobolev inequality then implies $u\in L^6(\R^3; \C^2)$. Moreover, by a straightforward modification of the proof of Lemma \ref{lem-mg-laplacian} 
we deduce that $u$ must satisfy
\begin{equation} 
\sigma\cdot (P-A) u = 0. 
\end{equation}
From \cite[Theorem.~2.1]{fl} we thus conclude that $u\in L^2(\R^3; \C^2)$,
as desired. Now Lemma \ref{lem-orthogonal} gives 
\begin{equation} 
0= \langle  u, W_P 
\begin{bsmallmatrix}1\\ 1\end{bsmallmatrix}
\rangle= - \langle  G_0\idd_2{w_P}^{\!\ast\,} f, W_P 
\begin{bsmallmatrix}1\\ 1\end{bsmallmatrix}
 \rangle = - \langle \, \U w_P G_0\idd_2 {w_P}^{\!\ast\,} f, w_P 
 \begin{bsmallmatrix}1\\ 1\end{bsmallmatrix}
\rangle.
\end{equation}
However, since $f\in \ker M_{P,0}$ we have $\U\,  w_P G_0\idd_2 {w_P}^{\!\ast\,} f= -f$. Hence 
$\langle  f, w_P \begin{bsmallmatrix}1\\ 1\end{bsmallmatrix} \rangle=0$, 
which implies $f\in \ker S_P M_{P,1} S_P$, 
cf.~equation \eqref{sm1s}. Hence  $\ker S_P M_{P,1} S_P=\ker M_{P,0}$.
\end{proof} 

The question of existence of zero modes (or zero energy eigenfunctions) of Pauli operators in dimension three is of current interest, see~\cite{be,bel,bvb,fl,fl2}. We will explore the connection to the results obtained  here. To do so we need to recall the set-up from~\cite{be,bel,bvb} in some detail, in order to define the quantity $\delta(B)$, see \eqref{delta-B} below.

Let $B$ satisfy Assumption~\ref{ass-B-2}. We consider the operators
\begin{equation}
H_P=(P-A)^2+\sigma\cdot B\quad\text{and}\quad
\widetilde{H}_P=(P-A)^2+\sigma\cdot B+\abs{B}=H_P+\abs{B}.
\end{equation}
They are obtained as the Friedrich extension of the corresponding forms with common form domain $\mathcal{Q}(H_P)=\mathcal{Q}(\widetilde{H}_P)
=H^{1,0}(\R^3;\C^2)$. In the quadratic form sense we have 
$\widetilde{H}_P\geq (P-A)^2$, since $\sigma\cdot B+\abs{B}\geq0$ as quadratic forms. 
It follows from Lemma~\ref{lem-mg-laplacian} that zero is a regular point of
$(P-A)^2$. As a consequence (cf.~Lemma~\ref{lem-ker-M0}) $\ran \widetilde{H}_P$ is dense in $\cH$. Let $\widetilde{\cH}$ be the completion of $\mathcal{Q}(H_P)$ under the norm
\begin{equation*}
\norm{u}_{\widetilde{\cH}}^2=\ip{u}{\widetilde{H}_Pu}.
\end{equation*}
The operators $\widetilde{H}_P^{\alpha}$ are defined for $\alpha=\pm\frac12$ and $\alpha=-1$ via the functional calculus, as self-adjoint operators on $\cH$. In particular, the operator
\begin{equation*}
\widetilde{H}_P^{-\frac12}\colon\ran{\widetilde{H}_P^{\frac12}}
\to \widetilde{\cH}
\end{equation*}
preserves norms. Since $\ran \widetilde{H}_P^{-\frac12}=\dom \widetilde{H}_P^{\frac12}=\mathcal{Q}(H_P)$ is dense in $\widetilde{\cH}$, the operator $\widetilde{H}_P^{-\frac12}$ extends to a unitary operator $\mathsf{U}\colon\cH\to\widetilde{\cH}$.

Due to the assumption on $B$ the multiplication by $\abs{B}^{\frac12}$ is bounded from $\widetilde{\cH}$ to $\cH$. Thus we can define $\sn=\abs{B}^{\frac12}\mathsf{U}\colon \cH\to\cH$, with the property
\begin{equation*}
\sn u=\abs{B}^{\frac12}\widetilde{H}_P^{-\frac12}u\quad
\text{for $u\in\ran \widetilde{H}_P^{\frac12}$}.
\end{equation*}
Then we define (see \cite[Equation (1)]{bvb})
\begin{equation}\label{delta-B}
\delta(B)=\inf\{\norm{(I-\sn^*\sn)f}\mid \norm{f}=1,\; \mathsf{U}f\in\cH\}.
\end{equation}

We recall some of the recent results on zero modes for $H_P$. Assuming that  $\abs{B}\in L^{3/2}(\R^3)$, Balinsky, Evans and Lewis  proved in \cite{bel} that if the operator $H_P$ has a zero eigenfunction, then $\delta(B) =0$. Later, Benguria and Van den Bosch proved the converse implication under the additional condition that $B$ satisfy equation \eqref{B-decay-cond} for some $\beta>1$, cf.~\cite[Theorem~1.1]{bvb}.  Finally, in \cite[Theorem~2.2]{fl}, Frank and Loss showed that the additional decay condition on $B$ introduced in \cite{bvb} is not necessary.

It is illustrative to verify that, under somewhat stronger assumptions on $B$, the identity $\delta(B) =0$ is equivalent to zero being an exceptional point  for $H_P$.

\begin{prop} \label{prop-zero-mode}
Let $B$ satisfy Assumption~\ref{ass-B-2} for some $\beta >3$. Then $\delta(B) =0$ if and only if zero is an exceptional point for $H_P$.
In the affirmative case the exceptional point is of the second kind.
\end{prop} 

\begin{proof}
Using Corollary~\ref{cor-regular-pauli} and $\abs{B(x)}^{\frac12}\lesssim \jap{x}^{-\beta/2}$ with $\beta>3$ we get that
\begin{equation}
\lim_{\eta\downarrow0}\abs{B}^{\frac12}
(\widetilde{H}_P+\eta I)^{-1}\abs{B}^{\frac12}=\abs{B}^{\frac12}
\widetilde{H}_P^{-1}\abs{B}^{\frac12}=\sn\sn^*,
\end{equation}
with convergence in operator norm.
We note that as a consequence the operator $\sn\sn^*$ is compact. But then $\sn^*\sn$ is also compact.

Assume that zero is an exceptional point for $H_P$. Then due to 
Lemmas~\ref{lem-ker-M0} 
and~\ref{lem-no-resonance}, there exists $u\in L^2(\R^3;\C^2)\cap  H^{1, -s}(\R^3;\C^2)$ with $\tfrac 12 < s < \beta -\tfrac 12$, such that 
$(\id + G_0\idd_2 W_P)u=0$, where $W_P= W_1 +\sigma\cdot B$, cf.~\eqref{W1-def}. Let $
\widetilde{W} \coloneqq W_P +\abs{B}$.

Hence
\begin{equation} 
(\id + G_0\idd_2 \widetilde{W})u= G_0\idd_2 \abs{B} u .
\end{equation} 
Since $(\id +G_0\idd_2 \widetilde W )^{-1}$
exists and is bounded on $H^{1,-s}(\R^3;\C^2)$ for any  $\tfrac 12 < s < \beta -\tfrac 12$, it follows that
\begin{equation} \label{eq-u} 
u = (\id + G_0\idd_2 \widetilde W )^{-1} G_0\idd_2  \abs{B} u .
\end{equation}
Let $f =\abs{B}^{\tfrac 12}  u$. Then $f\in L^2(\R^3;\C^2)$, and, in view of Corollary \ref{cor-regular-pauli}, 
\begin{align*} 
\sn\sn^* f &=  \abs{B}^{\tfrac 12}  \bigl[ \bigl(\sigma\cdot (P-A)\bigr)^2 + \abs{B}\bigr]^{-1}  \abs{B}^{\tfrac 12}  f =  \abs{B}^{\tfrac 12}  
(\id + G_0\idd_2 \widetilde W )^{-1}  G_0\idd_2 \abs{B} u \\
&=   \abs{B}^{\tfrac 12}  u = f.
\end{align*}
This shows that $\sn\sn^*$ has eigenvalue $1$, and therefore so does $\sn^* \sn$. Hence $\delta(B)=0$, see~\eqref{delta-B}.

Conversely, assume that $\delta(B)=0$. Then we can find a sequence $g_n\in\cH$ with $\norm{g_n}=1$ such that
\begin{equation*}
\lim_{n\to\infty}\norm{(\id-\sn^*\sn)g_n}=0.
\end{equation*}
It follows from Weyl's criterion that $1$ is in the spectrum  of $\sn^*\sn$. Since this operator is compact, $1$ is an eigenvalue of $\sn^*\sn$, hence also an eigenvalue of $\sn\sn^*$.
Thus there exists 
$f\in  L^2(\R^3;\C^2)$, $\norm{f}=1$, such that
\begin{equation} \label{f-ef}
 \abs{B}^{\tfrac 12} \widetilde{H}_P^{-1}  \abs{B}^{\tfrac 12} f = f.
\end{equation}
Since $ \abs{B}^{\tfrac 12} f\in H^{-1, \frac 32 +0}(\R^3;\C^2)$, it follows from Corollary~\ref{cor-regular-pauli} that 
$$
u \coloneqq \widetilde{H}_P^{-1}   \abs{B}^{\tfrac 12}
 f = (\id + G_0\idd_2 \widetilde{W})^{-1} G_0\idd_2 \abs{B}^{\tfrac 12} f \in  H^{1, -\frac 12 -0}(\R^3;\C^2).
$$
Moreover, from~\eqref{f-ef} we deduce the identity
\begin{align*}
\abs{B}^{\tfrac 12} f &= \abs{B}  
\widetilde{H}_P^{-1}  \abs{B}^{\tfrac 12} f = \abs{B}^{\tfrac 12}  f - H_P u, 
\end{align*} 
which implies $H_P u =0$.
This shows that $u\in \nul H_P$ and therefore $\ker M_0 \neq \{0\}$, cf.~Lemma~\ref{lem-ker-M0}. 

The last statement follows from Lemma~\ref{lem-no-resonance}.
\end{proof}

\begin{rem}
Sharp conditions for the nonexistence of zero energy eigenfunctions of $\sigma\cdot (P-A)$ in terms of $L^p$-norms of $B$ and $A$
were recently established in~\cite{fl,fl2}.
\end{rem}

\section{General perturbations}
\label{sec-general}
The set-up used here applies to a much larger class of perturbations of $-\Delta$ than the perturbations defined in~\eqref{W-def}, and leads to resolvent expansions as those obtained in previous sections.

The idea is to combine the factorization scheme in~\cite{jn} with some of the estimates from~\cite{JK}, extending what was done above. See also the comment on the bottom of page~588 in~\cite{JK}.

In the sequel we use the notation $\cH=L^2(\R^3)$ and $H_0=-\Delta$, with domain $H^{2,0}$.

\begin{assumption}\label{gen-V}
Let $\myW$ be a symmetric $H_0$-form-compact operator on $\cH$. 
Let $\beta>0$. Assume that $\myW$ defines a compact operator in $\B(1,-\beta/2;-1,\beta/2)$, also denoted by $\myW$. 
\end{assumption}

\noindent
Note that Proposition \ref{prop-M0} continues to hold in this abstract setting. In particular, the proof of the fact that $0$ is an isolated point of $\sigma(M_0)$ remains unchanged.

\medskip

\noindent We have the following result. The proof is a variant of the proof of \cite[Proposition~A.1]{IJ}.

\begin{lem}\label{lem-gen-fac}
Let $\myW$ satisfy Assumption~\ref{gen-V} for some $\beta>0$.
Let $\K=\ell^2(\N)$, if $\rank \myW=\infty$. Otherwise, let $\K=\C^{\rank \myW}$. Then there exist a bounded operator $\myw\colon H^{1,-\beta/2}\to\K$ and a self-adjoint and unitary operator
$\myU$ on $\K$ such that
\begin{equation}\label{gen-fac}
\myW=\myw^*\myU \myw.
\end{equation}
\end{lem}
\begin{proof}
We assume $\myW\neq 0$.
Define 
\begin{equation*}
\widetilde{\myW}=\jap{x}^{\beta/2}\jap{P}^{-1}\myW\jap{P}^{-1}\jap{x}^{\beta/2}.
\end{equation*}
By assumption this operator is compact and self-adjoint on $\cH$.

Let $N=\rank \myW$, and let $\{u_j\mid j=1,2,\ldots,N\}$, be an orthonormal sequence in $\cH$ such that
\begin{equation*}
\widetilde{\myW}=\sum_{j=1}^{N}\lambda_j\ket{u_j}\bra{u_j}.
\end{equation*}
Here $\{\lambda_j\}$ denotes the \emph{non-zero} eigenvalues of $\widetilde{\myW}$, repeated with multiplicity. 

Define $\myU$ on $\K$ as a matrix by
\begin{equation}
(\myU)_{mn}=\begin{cases}
\sgn(\lambda_m), & \text{for $m=n$, $1\leq m\leq N$},\\
0, & \text{otherwise}.
\end{cases}
\end{equation}
Then $\myU$ is self-adjoint and unitary. For $j=1,2,\ldots,N$, define
\begin{equation}
\eta_j=\jap{P}\jap{x}^{-\beta/2}u_j. 
\end{equation}
Then define $\myw\colon H^{1,-\beta/2}\to\K$ by
\begin{equation}
(\myw f)_j=\begin{cases}
\abs{\lambda_j}\ip{\eta_j}{f}, & j=1,2,\ldots, N,\\
0, & j>N,
\end{cases}
\end{equation}
for $f\in H^{1,-\beta/2}$.
With these definitions the factorization \eqref{gen-fac} follows.
\end{proof}

\begin{rem}
In explicit cases, e.g. the magnetic perturbation $W$ in~\eqref{W-def}, there are other factorizations that are `natural'. The same holds for a multiplicative perturbation. On the other hand, in the case of a self-adjoint  finite rank perturbation Lemma~\ref{lem-gen-fac} gives a natural factorization. In any case, due to the uniqueness of coefficients in an asymptotic expansion, the choice of factorization does not matter. However, it may be difficult to see explicitly in concrete examples that two coefficient expressions are equal.
\end{rem}

\begin{rem}
Recall that factored perturbations are additive in the following sense. Let $\myW_j$, $j=1,2$, be perturbations satisfying Assumption~\ref{gen-fac}. Let $\myW_j=\myw_j^*\myU_j\myw_j$, $j=1,2$, be factorizations with intermediate Hilbert spaces $\K_j$, $j=1,2$, and with the mapping properties stated in Lemma~\ref{lem-gen-fac}.

Let $\myW=\myW_1+\myW_2$ and $\K=\K_1\oplus\K_2$. Define
\begin{equation}
\myw=\begin{bmatrix} \myw_1 \\ \myw_2 \end{bmatrix}
\quad\text{and}\quad
\myU=\begin{bmatrix}
\myU_1 & 0 \\ 0  & \myU_2 \end{bmatrix}.
\end{equation} 
Here we use matrix notation for operators on $\K=\K_1\oplus\K_2$.
Then it is straightforward to verify that $\myW$ satisfies Assumption~\ref{gen-fac} and that we have the factorization 
$\myW=\myw^*\myU \myw$, with $\myw$ and $\myU$ having the mapping properties stated in Lemma~\ref{lem-gen-fac}.
\end{rem}

\appendix

\section{Computation of $(m_{P,0}+S_{P,1})^{-1}$}\label{appA}
In this appendix we give some of the details in the computation of 
$(m_{P,0}+S_{P,1})^{-1}$ and its application.
Consider the second part of case (E3), i.e. the vectors $S_P\alpha1$ and
$S_P\beta1$ are linearly independent. Note that we have the orthogonal direct sum decomposition
\begin{equation}\label{direct-sum}
S_P\K_P= S_{P,1}\K_P \oplus (S_P-S_{P,1})\K_P.
\end{equation}
In the case we consider here $\dim (S_P-S_{P,1})\K_P=2$.
It suffices to find the inverse of $m_{P,0}$, considered as a map in $(S_P-S_{P,1})\K_P$. Since $\{ S_P\alpha1,S_P\beta1\}$ is a basis of $(S_P-S_{P,1})\K_P$ this amounts to inverting a $2\times 2$ matrix. 

To simplify the notation we introduce the shorthand notation $a=S_P\alpha1$ and $b=S_P\beta1$. They form a basis (not necessarily orthogonal) of the two dimensional space $\widetilde{\K}=(S_P-S_{P,1})\K_P$, i.e.
$\widetilde{\K}=\myspan(\ket{a},\ket{b})$. To find the inverse of the map $m_{P,0}$ in $\widetilde{\K}$ we first find an expression for the identity, denoted by $I_{\widetilde{\K}}$, in terms of the four rank one operators $\ket{a}\bra{a}$, $\ket{a}\bra{b}$, $\ket{b}\bra{a}$, and $\ket{b}\bra{b}$. The result is
\begin{equation*}
I_{\widetilde{\K}}=
\frac{1}{\norm{a}^2\norm{b}^2
-\abs{\ip{a}{b}}^2}
\Bigl[ \norm{b}^2\ket{a}\bra{a}
- \ip{a}{b}\ket{a}\bra{b}
-\ip{b}{a}\ket{b}\bra{a}
+\norm{a}^2\ket{b}\bra{b}
\Bigr].
\end{equation*}
Next we solve the equation
\begin{equation*}
\bigl(c_1\ket{a}\bra{a}+c_2\ket{a}\bra{b}
+c_3\ket{b}\bra{a}+c_4\ket{b}\bra{b}\bigl)m_{P,0}=I_{\widetilde{\K}}.
\end{equation*}

\noindent The solution is 
\begin{equation*}
(c_1, c_2, c_3, c_4) =\kappa\, \big( \norm{b}^4+\abs{\ip{a}{b}}^2, - (\norm{b}^2+\norm{a}^2) \ip{a}{b}, -( \norm{b}^2+\norm{a}^2)\ip{b}{a}, \norm{a}^4+\abs{\ip{a}{b}}^2 \big),
\end{equation*} 
where $\kappa =  \frac{-4\pi}{(\norm{a}^2\norm{b}^2\
-\abs{\ip{a}{b}}^2)^2}$.
This   implies 
\begin{equation*}
(m_{P,0}\bigl\lvert_{\widetilde{\K}})^{-1} = 
\kappa
\begin{bmatrix}
\ket{a} & \ket{b}
\end{bmatrix}
\begin{bmatrix}
\norm{b}^4+\abs{\ip{a}{b}}^2 & -(\norm{a}^2+\norm{b}^2)\ip{a}{b}
\\
-(\norm{a}^2+\norm{b}^2)\ip{b}{a} & (\norm{a}^4+\abs{\ip{a}{b}}^2)
\end{bmatrix}
\begin{bmatrix}
\bra{a} \\ \bra{b}
\end{bmatrix}\, .
\end{equation*}
The $2\times2$ matrix above is self-adjoint, so it can be diagonalized. This implies that we can find $e,f\in \widetilde{\K}$ such
\begin{equation}\label{mP0-inverse}
(m_{P,0}\bigl\lvert_{\widetilde{\K}})^{-1}=\ket{e}\bra{e}+\ket{f}\bra{f}.
\end{equation}
Due to the decomposition~\eqref{direct-sum} we have
\begin{equation}\label{m0inverse}
(m_{P,0}+S_{P,1})^{-1}=S_{P,1}+ (m_{P,0}\bigl\lvert_{\widetilde{\K}})^{-1}.
\end{equation}
The next step is to use \eqref{mP0-inverse} and \eqref{m0inverse} in \cite[(A.36)]{jn1}.    Following the calculations in \cite[Appendix A]{jn1} we get
\begin{equation}\label{F-1-new}
F_{-1}=P_0W_PG_3\idd_2W_PP_0+\ket{\psi_c^1}\bra{\psi_c^1}
+\ket{\psi_c^2}\bra{\psi_c^2},
\end{equation} 
where $\psi_c^1,\psi_c^2\in\nul(H_P)$ are give by 
\begin{equation} \label{psi-12}
\begin{aligned}
\ket{\psi_c^1} &= \frac{1}{\|e\|^2} \big( P_0W_PG_2\idd_2 w^* \ket{e}\,  -  G_0\idd_2 w^* \ket{e} \big) \\[7pt]
\ket{\psi_c^2} &= \frac{1}{\|f\|^2} \big( P_0W_PG_2\idd_2 w^* \ket{f}\,  -  G_0\idd_2 w^* \ket{f} \big)\, .
\end{aligned}
\end{equation}

The case (E1) with multiplicity two zero resonances can be treated in the same manner. We obtain formulas \eqref{F-1-new} and \eqref{psi-12} with $P_0=0$.

\paragraph*{Acknowledgments.} AJ was partially supported by grant 8021--00084B from Independent Research Fund Denmark \textbar\ Natural Sciences.


\begin{thebibliography}{99}

\bibitem{be} A.~Balinsky and W.~D.~Evans: On the zero modes of Pauli operators. {\em J. Func. Anal.} {\bf 179} (2001) 120--135.

\bibitem{bel} A.~Balinsky, W.~D.~Evans and R.~T.~Lewis: Sobolev, Hardy and CLR inequalities associated with Pauli operators in $\R^3$. {\em J. Phys. A.} {\bf 34} (2001) L19. 


\bibitem{bvb} R.~Benguria and H. Van Den Bosch: A criterion for the existence of zero modes for the Pauli operator with fastly decaying fields. {\em J.~Math.~Phys.} {\bf 56} (2015) pp. 052104

\bibitem{egs} M. B. Erdo\u{g}an, M. Goldberg and W. Schlag:
Strichartz and smoothing estimates for Schr\"{o}dinger
operators with large magnetic potentials in $\R^3$.
\emph{J. Eur. Math. Soc.} \textbf{10} (2008) 507--531.

\bibitem{fl} R.~L.~Frank and M.~Loss: Which magnetic fields support a zero mode? {\em J. Reine Angew. Math.} {\bf 788} (2022) 1--36. 

\bibitem{fl2} R.L.~Frank and M.~Loss: A sharp criterion for zero modes of the Dirac equation. {\em  J. Eur. Math. Soc.}, to appear.   

\bibitem{elt} D.~M.~Elton: The local structure of zero mode producing magnetic potentials. {\em Comm. Math. Phys}. {\bf 229} (2002) 121--139.

\bibitem{IJ} K. Ito and A. Jensen: Resolvent expansion for the Schrödinger operator on a graph with infinite rays. \emph{J. Math. Anal.} \textbf{464} (2018) 616--661.

\bibitem{JK} A. Jensen and T. Kato: Spectral properties of Schr\"odinger operators and time-decay of the wave functions. {\em Duke Math. J.} {\bf 46} (1979) 583--611.

\bibitem{jn} A.~Jensen and G.~Nenciu: A unified approach to resolvent expansions at thresholds. {\em Rev.~Math.~Phys.} {\bf 13} (2001) 717--754.

\bibitem{jn2} A.~Jensen and G.~Nenciu: Erratum: ``A unified approach to resolvent expansions at thresholds''. {\em Rev.~Math.~Phys.} {\bf 16} (2004) 675--677.

\bibitem{jn3} A.~Jensen and G.~Nenciu: Schr\"odinger operators on the half line: Resolvent expansions and the Fermi golden rule at thresholds. {\em Proc. Indian Acad. Sci. (Math. Sci.)} {\bf 116} (2006)
375--392.

\bibitem{jn1} A.~Jensen and G.~Nenciu: The Fermi Golden Rule and its form at thresholds in odd dimensions. {\em Commun. Math. Phys.} {\bf 261} (2006) 693--727.

\bibitem{kk} A.~I. Komech and E.~A. Kopylova: Dispersive decay for the magnetic Schrödinger equation. \emph{J. Funct. Anal.} \textbf{264} (2013)
735--751.

\bibitem{kov1} H.~Kova\v r\'{\i}k: Resolvent expansion and time decay of the wave functions for two-dimensional magnetic Schr\"odinger operators. {\em Comm.~Math.~Phys.} {\bf 337} (2015) 681--726. 

\bibitem{kov2}  H.~Kova\v r\'{\i}k: Spectral properties and time decay of the wave functions of Pauli and Dirac operators in dimension two.  {\em Adv. Math.} {\bf 398} (2022), 108244.


\bibitem{LL} E. H. Lieb and M. Loss, \textit{Analysis. Second edition}. Graduate Studies in Mathematics, 14. American Mathematical Society, Providence, RI, 2001. xxii+346 pp.

\bibitem{ly} M. Loss and H.-T. Yau: Stability of Coulomb systems with magnetic fields. III. Zero
energy bound states of the Pauli operator. {\em  Comm. Math. Phys.} {\bf 104} (1986) 282--290.

\bibitem{mu} M. Murata: Asymptotic expansions in time for solutions of Schr\"odinger-type equations. {\em J. Funct. Anal.} {\bf 49} (1982), 10--56.

\bibitem{si} B. Simon: Kato’s inequality and the comparison of semigroups. {\em J. Funct. Anal.} {\bf 32} (1979), 97--101.

\bibitem{ya} D. R. Yafaev: Scattering by magnetic fields. {\em  St. Petersburg Math. J.} {\bf 17} (2006) 875--895.

\end{thebibliography}
\end{document}